\documentclass[11pt]{article}

\usepackage[margin=1in]{geometry}  % set the margins to 1in on all sides
\usepackage{graphicx}              % to include figures
\usepackage{amsmath}               % great math stuff
\usepackage{amsfonts}              % for blackboard bold, etc
\usepackage{amsthm}                % better theorem environments
\usepackage{caption}
\usepackage{enumerate}

\usepackage{amssymb}
\usepackage{amscd} 
\usepackage{epsfig}

\usepackage{hyperref}
\newtheorem{thm}{Theorem}[section]
\newtheorem{lem}[thm]{Lemma}

\newcommand{\RR}{\mathbb{R}}      % for Real numbers
\newcommand{\vb}[3]{\begin{pmatrix}#1\\\\#2\\\\#3\end{pmatrix}}

\numberwithin{equation}{section}

\begin{document}

\nocite{*}

\title{On stochastic perturbations of dynamical systems with a ``rough" symmetry. Hierarchy of Markov chains.}

\author{Mark Freidlin  \\
Department of Mathematics\\
University of Maryland\\
College Park, Maryland 20742-4015}

\date{}

\maketitle

\begin{abstract}
We consider long-time behavior of dynamical systems perturbed by a small noise. Under certain conditions, a slow component of such a motion, which is most important for long-time evolution, can be described as a motion on the cone of invariant measures of the non-perturbed system. The case of a finite number of extreme points of the cone is  considered in this paper. As is known, in the generic case, the long-time evolution can be described by a hierarchy of cycles defined by the action functional for corresponding stochastic processes. This, in particular, allows to study metastable distributions and such effects as stochastic resonance. If the system has some symmetry in the logarithmic asymptotics of transition probabilities (rough symmetry),the hierarchy of cycles should be replaced by a hierarchy of Markov chains and their invariant measures.
\end{abstract}

\section{Limiting motion on the cone of invariant measures}

Together with a dynamical system
\begin{equation}\label{01}
\dot{X}_t=b(X_t), \quad X_0 =x \in \RR^n, 
\end{equation}
in $\RR^n$, consider a perturbed system 
\begin{equation}\label{02}
\dot{X}^{\epsilon,\delta}_t=b(X^{\epsilon,\delta}_t)+\sqrt{\epsilon}\sigma(X^{\epsilon,\delta}_t)\dot{W}_t+\delta\beta(X^{\epsilon,\delta}_t), \quad X^{\epsilon,\delta}_0 =x \in \RR^n.
\end{equation}
Here $W_t$ is the Wiener process in $\RR^n$, $\sigma(x)$ is  $n \times n$-matrix, $b(x)$ and $\beta(x)$ are vector fields, $0 \le \epsilon \ll 1$, $0 \le \delta \ll 1$. We assume that components of the vector fields $b(x)$ and $\beta(x)$ have continuous bounded derivatives as well as the entries of the diffusion matrix $a(x)=(a_{ij}(x))=\sigma(x)\sigma^*(x)$; $\det a(x)>0$. We are interested in the long time behavior of the perturbed system when at least one of the parameters $\epsilon$ and $\delta$ is small. The invariant measure of a system is the simplest characteristic of its long-time behavior. Even if there exists a unique normalized invariant measure of system \eqref{01}, the perturbed system can have many normalized invariant measures. But, under mild additional assumptions, all these measures converge in the weak topology to the (unique) invariant measure of the original system as perturbations tend to zero. Thus, in this case, the long-time behavior of the perturbed system is, roughly speaking, close to the behavior of the non-perturbed system.

Situation is different if the original system has multiple invariant measures. Roughly speaking, the perturbed system, first, approaches the invariant measure ``closest" to the initial point, but then, because of perturbations, it moves, in general, to the support of another invariant measure. The main characteristic of long-time behavior of the perturbed system is given by the slow component evolution which is, actually, a motion on the set of invariant measures of the original system. This set forms a cone. In many important problems, the slow evolution converges, in an appropriate time scale, to a limiting motion. The position of the limiting motion at the cone together with the invariant measure corresponding to this position are main characteristics of the long-time evolution.

As is known (The Krein-Milman theorem \cite{06}), each cone is a convex envelope of its extreme points. So that to describe the whole cone, it is sufficient to parametrize the extreme points set. Then any point of the cone can be characterized by a measure on the extreme points set. Such an approach is convenient in many perturbation theory problems for both deterministic and stochastic perturbations. Say, if \eqref{01} is a one degree of freedom Hamiltonian system with one well Hamiltonian $H(x)$, $x \in \RR^2$, $\lim_{|x| \rightarrow \infty} H(x)=\infty$, extreme invariant measures can be parametrized by the value of the energy: on each non-empty set $\{x \in \RR^2 : H(x)=h\}$, just one normalized invariant measure is concentrated. Slow component at the long-time evolution of the perturbed system in this case is described by the classical averaging principle. If the Hamiltonian has several wells, the set of extreme normalized invariant measures can be parametrized by points of a graph \cite{01,05}.

We  consider in this paper the simplest case when system \eqref{01} has just a finite number of normalized extreme invariant measures. For example, this will be the case when \eqref{01} has a finite number of asymptotically stable equilibriums or limit cycles $K_1,...,K_l$ and each point of $\RR^n$, besides the separatrices, is attracted to one of $K_i$ (invariant measures concentrated on separatrix surfaces are not important in this case). Then the trajectory started at a point $x$ belonging to the basin of $K_{i(x)}$, first, approaches $K_{i(x)}$. We assume that the diffusion matrix $a(x)$ is non-degenerate, so that the perturbed trajectory, sooner or later, leaves the basin of $K_{i(x)}$ and switches to another attractor $K_j$, then to another, and so on. The transitions occur due to the diffusion term in \eqref{02}. The drift perturbation $\delta\beta(x)$ for small enough $\delta \ge 0$ is not essential for the transitions. So that, for brevity, we omit the deterministic perturbation and consider the perturbed system
\begin{equation}\label{03}
\dot{X}^\epsilon_t=b(X^\epsilon_t)+\sqrt{\epsilon}\sigma(X^\epsilon_t)\dot{W}_t,\quad X^\epsilon_0=x \in \RR^n.
\end{equation}

Denote by $\frac{1}{\epsilon}S_{0T}(\varphi)$ the action functional \cite{05} for the family of processes $X^\epsilon_t$ as $\epsilon \downarrow 0$ in the space $C_{0T}$ of continuous functions $\varphi: [0,T] \mapsto \RR^n$. For absolutely continuous $\varphi \in C_{0T}$,
\[S_{0T}(\varphi)=\frac{1}{2}\int_0^T\left(a^{-1}(\varphi_s)(\dot{\varphi}_s-b(\varphi_s))\cdot (\dot{\varphi}_s-b(\varphi_s))\right)ds.\]

Let $V(x,y)=\inf\{S_{0T}(\varphi): \varphi \in C_{0T}, \varphi_0=x, \varphi_T=y, T \ge 0\}$, $x, y \in \RR^n$. If $b(x)=-\nabla U(x)$ and $a(x)$ is the unit matrix, $V(x,y)$ can be expressed through the potential $U(x)$. If $O$ is an asymptotically stable equilibrium, the function $V(y)=V(O,y)$ is called quasi-potential (with respect to the point $O$ and perturbations defined by the matrix $a(x)$). As is known (see \cite{05}), asymptotics of many interesting characteristics of system \eqref{03} can be expressed through the function $V(x,y)$. In particular, if system \eqref{01} has a finite number of asymptotically stable equilibriums $K_1,...,K_l$ (or asymptotically stable attractors of more general form supporting a unique invariant measure (see Ch.6 in \cite{05}) and each $x \in \RR^n$ besides the points of the separatrix surfaces is attracted to one of $K_i$, then the asymptotic behavior of the invariant measure $\mu^\epsilon$ of the process $X^\epsilon_t$ defined by \eqref{03} can be described through the numbers $V_{ij}=V(K_i,K_j)$ (\cite{05}, Ch.6). Under certain additional assumptions, the metastable states (distributions) for various initial points and time scales can be found using the numbers $V_{ij}$, $i,j \in \{1,...,l\}$. The assumptions, roughly speaking, are:
(1) for any $x \in \RR^n$ not belonging to a separatrix a unique $K_{i(x)}$ exists such that $V(x,K_{i(x)})=0$ and (2) minima of certain finite sums of numbers $V_{ij}$ over some finite sets are achieved at a unique point of this set (see Section 3). In particular, it is assumed that $\min_{j:j \neq i} V_{ij}=V_{ij^*(i)}$ for each $i \in \{1,...,l\}$ is achieved at one point $j^*(i)$. These assumptions are in a sense robust, but there are interesting examples where they are not satisfied. For instance, an exact symmetry in \eqref{03} can lead to non-uniqueness. Actually, one can have non-uniqueness of $i(x)$ which is preserved for all diffusion matrices $a(x)$ (see the next Section).

The numbers $V_{ij}$ define logarithmic asymptotics of certain characteristics of system \eqref{03} as $\epsilon \downarrow 0$. In particular, the equality $\min_{j: j \neq i} V_{ij}=V_{ij_1}=V_{ij_2}$ guarantees that logarithmic asymptotics as $\epsilon \downarrow 0$ of probabilities of events $\mathcal{E}_m=\{$first exit from the domain of attraction of $K_i$ will be to the domain of $K_{j_m}\}$, $m \in \{1,2\}$, coincide. But these probabilities are not equivalent as in the case of exact symmetry. This is why we call it rough or log-asymptotic symmetry. It turns out that many important characteristics of system \eqref{03} as $0 <\epsilon \ll 1$ still can be calculated using the numbers $V_{ij}$ even if there are rough symmetries. To calculate the asymptotics of other characteristics for a system with rough symmetries, it is not enough just to know $V_{ij}$: one should take into account the pre-exponential factors.

In the next Section, we consider a system with rough symmetry and calculate all metastable states (distributions). In Section 3, we describe a general construction of a hierarchy of Markov chains which is a generalization of the hierarchy of cycles \cite{03,04,05} introduced for system without rough symmetries.

\section{Nearly-Hamiltonian system with rough symmetry}

Consider a dynamical system in $\RR^2$
\begin{equation}\label{04}
\dot{X}_t=b(X_t),\quad X_0=x \in \RR^2,
\end{equation}
shown in Figure 1.

\begin{figure}[h]
\centering
\begin{picture}(310, 210)
    \put(0,0){\includegraphics[scale=0.6]{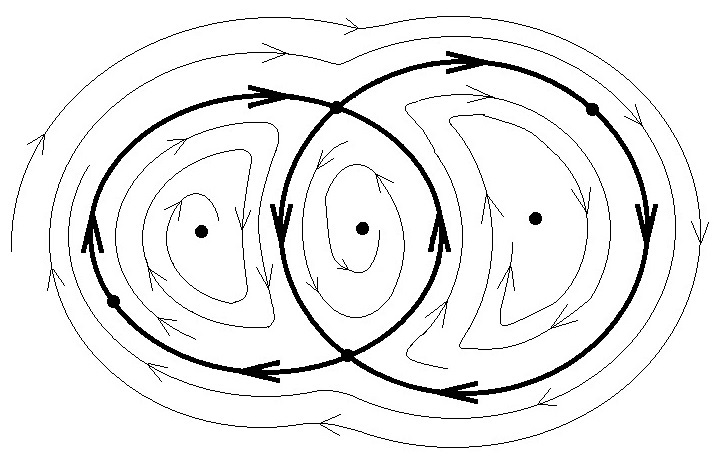}}
	\put(90,90){$O_1$}	
	\put(163,92){$O_2$}	
	\put(242,96){$O_3$}	
	\put(144,166){$O_4$}	
	\put(153,33){$O_5$}
	\put(42,57){$M$}
	\put(271,160){$N$}
	\put(105,50){\Large $\bf G_1$}	
	\put(155,60){\Large $\bf G_2$}	
	\put(210,46){\Large $\bf G_3$}	
	\put(235,185){\Large $\bf G_4$}			
\end{picture}
\caption*{Figure 1}
\end{figure}

The system has two saddle points $O_4$ and $O_5$ and three asymptotically stable equilibriums $O_1$, $O_2$, and $O_3$. The separatrices connecting the saddle points divide the plane into four domains: $G_i$ attracted to $O_i$, $i=1,2,3$, and the exterior domain $G_4$ attracted to the curve $O_5MO_4N$ formed by two separatrices. Assume for brevity that outside a (large enough) circle centered at the origin $O$, the projection of $b(x)$ on the radius $Ox$ is directed to the origin and the length of the projection is bounded from below by a positive number. Together with the system \eqref{04}, we consider its perturbation \eqref{03}. The last assumption provides the positive recurrency of the process $X^\epsilon_t$ defined by \eqref{03} as well as the existence and uniqueness of a stationary distribution $\mu^\epsilon$ of the diffusion process \eqref{03}.

Denote by $\Pi$ the union of four separatrices connecting $O_4$ and $O_5$. Let 
\[
\alpha_i=\inf\{S_{0T}(\varphi) : \varphi \in C_{0T}, \varphi_0=O_i, \varphi_T \in \Pi, T \ge 0 \}.\]

Assume that all $\alpha_i$, $i \in \{1,2,3\}$, are different, say, $\alpha_1 <\alpha_2<\alpha_3$. It is easy to see that $V(x,y)=0$ for $x,y \in \Pi$ (if $V(x,y)=V(y,x)=0$, points $x$ and $y$ are called equivalent \cite{05}). Note that the equivalency of $x$ and $y$ is the same for various (non-degenerate) diffusion matrices $a(x)$. Taking into account that $V(x,O_i)=0$ for $x \in \Pi$, $i \in \{1,2,3\}$, we conclude that $V_{12}=V_{13}=\alpha_1$, $V_{21}=V_{23}=\alpha_2$, $V_{31}=V_{32}=\alpha_3$, so that our system has a log-asymptotic symmetry. The equalities $V_{12}=V_{13}$, $V_{21}=V_{23}$, $V_{31}=V_{32}$ (and log-asymptotic symmetry) are preserved for various matrices $a(x)$. Constants $\alpha_i$, $i\in\{1,2,3\}$ depend continuously on $a(x)$. 

In spite of the rough symmetry, the logarithmic asymptotics of the invariant measure $\mu^\epsilon$ of the process $X^\epsilon_t$ is defined by the numbers $V_{ij}=\alpha_i$: for each small enough neighborhood $\mathcal{E}_i$ of $O_i$, 
\begin{equation}\label{05}
\lim_{\epsilon \downarrow 0} \epsilon \ln \mu^\epsilon(\mathcal{E}_i)=-(\alpha_i-\alpha_1),\quad i \in \{1,2,3\}.
\end{equation}

This can be derived, for instance, from the fact that the exit time of the process $X^\epsilon_t$ from $G_i$, $i \in \{1,2,3\}$, is logarithmically equivalent to $\exp\{\frac{\alpha_i}{\epsilon}\}$ (see Section 5.5 in \cite{05}). The logarithmic asymptotics of $\mu^\epsilon(B)$ as $\epsilon \downarrow 0$ can be expressed through the function $V(x,y)$ for any Borel set $B \subset \RR^2$.

One can check that, starting from any point $x \in \RR^2$, for each $\delta >0$,
\[\lim_{\epsilon \downarrow 0} P_x\{|X^\epsilon_{T_\lambda(\epsilon)}-O_3|>\delta\}=0\]
if $\lim_{\epsilon \downarrow 0} \epsilon \ln T_\lambda(\epsilon)=\lambda>\alpha_2$. This means that $O_3$ is the metastable state for any $x \in \RR^2$ and time scale $T_\lambda(\epsilon)$ with $\lambda>\alpha_2$. If $0 <\lambda <\alpha_1$, the metastable state is $O_i$ if the initial point $x \in G_i$, $i \in \{1,2,3\}$. But if $0<\lambda<\alpha_1$ and the initial point $x \in G_4$, one can first say that $X^\epsilon_{T_\lambda(\epsilon)}$ is situated in the union of neighborhoods of $O_1$, $O_2$, and $O_3$ as $\epsilon \downarrow 0$. If $\lambda \in (\alpha_1, \alpha_2)$ and $x \in G_1\cup G_4$, the limiting distribution of $X^\epsilon_{T_\lambda(\epsilon)}$ as $\epsilon \downarrow 0$ is concentrated near $O_2$ and $O_3$. The exact limiting distribution between these neighborhoods is not defined by the numbers $V_{ij}$. To calculate the exact limiting distribution of $X^\epsilon_{T_\lambda(\epsilon)}$ between stable equilibriums of the non-perturbed system, we consider vector fields  $b(x)$ which are close to Hamiltonian fields.

Consider a Hamiltonian system $\dot{\widetilde{X}}_t=\overline{\nabla}H(\widetilde{X}_t)$ in $\RR^2$. Let the zero level set of the Hamiltonian $H(x)$ coincides with the set $\Pi$ of separatrices in Figure 1: $\Pi=\{x \in \RR^2: H(x)=0\}$; $H(x)$ is assumed to be smooth, positive in $G_2 \cup G_4$, and negative in $G_1 \cup G_3$. The function $H(x)$ has saddle points at $O_4$ and $O_5$ and extremums at $O_1$, $O_2$, and $O_3$. It is assumed that $H(x)$ has no other critical points. The trajectories of this Hamiltonian system are shown in Figure $2^a$.

\begin{figure}[h]
\centering
\begin{picture}(500,185)
    \put(0,10){\includegraphics[scale=0.5]{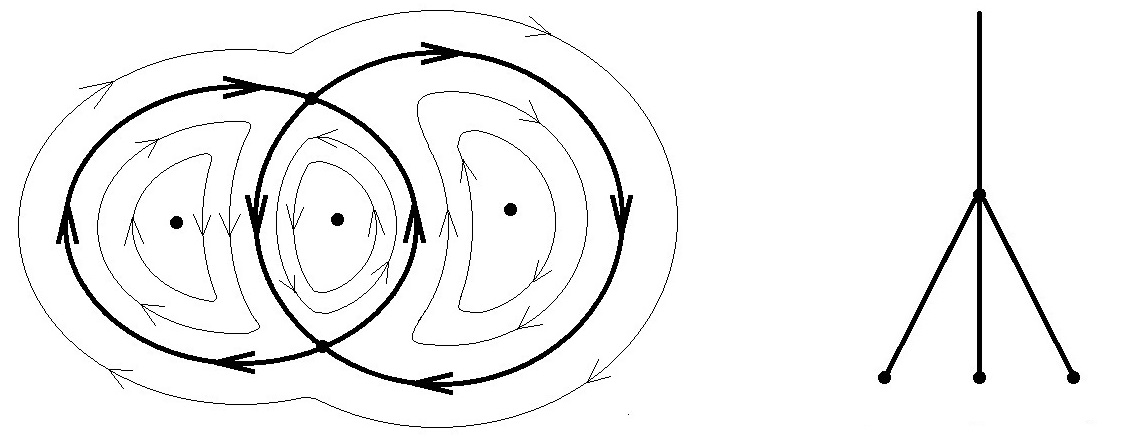}}
	\put(120,0){(a)}
	\put(360,0){(b)}	
	\put(63,97){$O_1$}	
	\put(124,97){$O_2$}	
	\put(190,100){$O_3$}	
	\put(110,145){$O_4$}	
	\put(115,29){$O_5$}
	\put(85,50){\Large $\bf G_1$}	
	\put(116,58){\Large $\bf G_2$}	
	\put(150,46){\Large $\bf G_3$}	
	\put(220,145){\Large $\bf G_4$}
			
	\put(330,57){\Large $\bf I_1$}			
	\put(355,57){\Large $\bf I_2$}			
	\put(395,57){\Large $\bf I_3$}			
	\put(373,130){\Large $\bf I_4$}				
	\put(374,98){\Large $\bf O$}			
	\put(330,17){\Large $\bf O_1$}			
	\put(365,17){\Large $\bf O_2$}			
	\put(402,17){\Large $\bf O_3$}
	\put(400,100){\Huge $\bf \Gamma$}
  \end{picture}
\caption*{Figure 2}
\end{figure}

Put $C_i(z)=\{x \in G_i : H(x)=z\}$; some of $C_i(z)$ are empty. The dynamical system $\widetilde{X}_t$ has a unique normalized invariant measure $\mu_{z,i}$ on each (non-empty) $C_i(z)$; the density $m_{z,i}(x)$ of $\mu_{z,i}$ on $C_i(z)$ is equal to 
\[
\left(\oint_{C_i(z)} \frac{dl}{|\nabla H(y)|} \cdot |\nabla H(x)|\right)^{-1}, \quad x \in C_i(z).
\] 

The extreme measures can be parametrized by points of the graph $\Gamma$ shown in Figure $2^b$. Each edge $I_i$, $i \in \{1,2,3,4\}$, corresponds to measures concentrated on periodic trajectories situated inside domain $G_i$. Vertices $O_1$, $O_2$, and $O_3$ correspond to stable equilibriums; one normalized invariant measure is concentrated at each equilibrium. The vertex $O$, actually, corresponds to two extreme measures: one concentrated at $O_4$ and another at $O_5$.

The pairs $(z,i)$, $z \in [\inf H(x), \sup H(x))$, $i \in \{1,2,3,4\}$, form a global coordinates system on $\Gamma$: a point $y \in \Gamma$ corresponding to the measure concentrated on $C_i(z)$ has coordinates $(z,i)$.

Consider now a deterministic perturbation of system $\widetilde{X}_t$:
\begin{equation}\label{06}
\dot{\widetilde{X}^\delta_t}=\overline{\nabla}H(\widetilde{X}^\delta_t)+\delta\beta(\widetilde{X}^\delta_t), \quad \widetilde{X}^\delta_0=x \in \RR^2.
\end{equation}
Here, $\beta(x)$ is a smooth bounded vector field vanishing on the set $\Pi=\{x \in \RR^2: H(x)=0\}$ such that $(\nabla H(x),\beta(x))>0$ for $x \in G_2$ and $(\nabla H(x),\beta(x))<0$ for $x \in G_1 \cup G_3 \cup G_4$. These assumptions imply that trajectories $\widetilde{X}^\delta_t$ behave as is shown in Figure 1:  $\widetilde{X}^\delta_t$ is attracted to $O_i$ if $\widetilde{X}^\delta_0=x \in G_i$, $i \in \{1,2,3\}$, and $\widetilde{X}^\delta_t$ is attracted to the curve $\Pi$ if $\widetilde{X}^\delta_0=x \in G_4$.

We consider in this Section stochastic perturbations of system \eqref{06}. The strength $\epsilon$ of stochastic perturbations is assumed to be much smaller than $\delta$ so that $\epsilon=\kappa\delta$, $0 <\kappa\ll 1$. The stochastically perturbed system has the form
\begin{equation}\label{07}
\dot{\widetilde{X}}^{\delta,\kappa}_t=\overline{\nabla}H(\widetilde{X}^{\delta,\kappa}_t)+\delta\beta(\widetilde{X}^{\delta,\kappa}_t)+\sqrt{\kappa\delta}\sigma(\widetilde{X}^{\delta,\kappa}_t)\dot{W}_t.
\end{equation}

We are interested in the transitions between asymptotically stable equilibriums of system \eqref{06}. The trajectories of \eqref{06} approach the equilibriums, roughly speaking, with the rate $\delta$. Therefore, it is natural to consider the time change $X^{\delta,\kappa}_t=\widetilde{X}^{\delta,\kappa}_{\frac{t}{\delta}}$. Then $X^{\delta,\kappa}_t$ satisfies the equation
\begin{equation}\label{08}
\dot{X}^{\delta,\kappa}_t=\frac{1}{\delta}\overline{\nabla}H(X^{\delta,\kappa}_t)+\beta(X^{\delta,\kappa}_t)+\sqrt{\kappa}\sigma(X^{\delta,\kappa}_t)\dot{W}_t.
\end{equation}

For each fixed $\kappa>0$ and $\delta \ll 1$, the motion described by \eqref{08} has two components: the fast one which is, roughly speaking, the rotation along periodic trajectories of the Hamiltonian system, and the slow component which is transversal to Hamiltonian trajectories. The fast component can be characterized by the invariant measures on those periodic trajectories. Each such measure is an extreme point of the cone $\mathcal{K}$ of invariant measures of system $\widetilde{X}_t$. The slow component is a motion on $\mathcal{K}$. 

Denote by $\bf Y$ the projection of $\RR^2$ on the graph $\Gamma$ which parametrizes the extreme points of $\mathcal{K}$: ${\bf Y}(x)$, $x \in \RR^2$, is the point of $\Gamma$ with coordinates $(z(x),i(x))$ such that $x \in C_i(z)$. The process $Y^{\delta,\kappa}_t={\bf Y}(X^{\delta,\kappa}_t)$ is the slow component of $X^{\delta,\kappa}_t$ (compare with Chapter 8 of \cite{05}). The process $Y^{\delta,\kappa}_t$, in general, is not Markovian.

Now, equation \eqref{08} has form \eqref{03} with $b(x)=b^\delta(x)=\frac{1}{\delta}\overline{\nabla}H(x)+\beta(x)$ and $\epsilon=\kappa$. The numbers $V_{ij}$ depend on $\delta$:
\begin{align*}
V_{ij}=V_{ij}^\delta=&\inf\left\{\frac{1}{2}\int_0^T\left(a^{-1}(\varphi_s)(\dot{\varphi}-\frac{1}{\delta}\overline{\nabla}H(\varphi_s)-\beta(\varphi_s)),(\dot{\varphi}-\frac{1}{\delta}\overline{\nabla}H(\varphi_s)-\beta(\varphi_s))\right)ds :\right.\\
& \varphi \in C_{0T}: \varphi_0=O_i, \varphi_T=O_j, T\ge 0 \bigg\}, \quad a(x)=\sigma(x)\sigma^*(x).
\end{align*}
Again we have $V^\delta_{12}=V^\delta_{13}=\alpha^\delta_1$, $V^\delta_{21}=V^\delta_{23}=\alpha^\delta_2$, $V^\delta_{31}=V^\delta_{32}=\alpha^\delta_3$. So that the system has rough symmetry.

Put 
\begin{align}\label{09}
&\hat{a}_i(z)=\int_{G_i(z)} \text{div}(a(x)\nabla H(x))dx, \quad T_i(z)=\oint_{C_i(z)}\frac{dl}{|\nabla H(x)|},\\
&\notag \hat{\beta}_i(z)=\int_{G_i(z)} \text{div}\beta(x)dx, \quad a_i(z)=\frac{\hat{a}_i(z)}{T_i(z)}, \quad \beta_i(z)=\frac{\hat{\beta}_i(z)}{T_i(z)}.
\end{align}
Here, $G_i(z)$ is the domain in $\RR^2$ bounded by $C_i(z)$. Define differential operators $L^\kappa_i$, $i \in \{1,2,3,4\}$:
\[
L^\kappa_i u(z)=\frac{\kappa}{2}a_i(z)\frac{d^2u}{dz^2}+\beta_i(z)\frac{du(z)}{dz}.
\]
Let 
\[a_i=a_i(O)=\int_{G_i}\text{div}(a(x)\nabla H(x))dx.\]

Processes $Y^{\delta,\kappa}_t$ on the graph $\Gamma$ converge weakly in the space of continuous functions $\varphi: [0,T]\mapsto \Gamma$ as $\delta \downarrow 0$ to a diffusion process $Y^\kappa_t$ on $\Gamma$. The process $Y^\kappa_t$ can be described by its generator $A^\kappa$:
\[A^\kappa f(z,i)=L^\kappa_i f(z,i)+\kappa h_i(z)\frac{df(z,i)}{dz}\]
for $z \in I_i$; the additional term $\kappa h_i(z)\frac{d}{dz}$ is a result of averaging of It\^{o}'s correction term in expression for $H(X^{\delta,\kappa}_t)$. Since we are interested in the limit $\kappa \downarrow 0$, this term will not influence the result, and we will not write the exact expression for $h_i(z)$. The domain of definition $D_{A^\kappa}$ of $A^\kappa$ consists of continuous on $\Gamma$ and smooth enough inside the edges functions $f(z,i)$ on $\Gamma$ such that the function $L^\kappa_if(z,i)$ is continuous on $\Gamma$ and satisfies the following gluing condition:
\[
\left.\sum_{i=1}^4 (-1)^i a_i(O)D_if(z,i)\right|_{(z,i)=O}=0.
\]
Here, $D_i$ is the operator of differentiation along the edge $I_i$. These conditions define the process $Y^\kappa_t$ in a unique way (Ch.8 in \cite{05}).

To prove the convergence of $Y^{\delta,\kappa}_t$, $0 \le t \le T<\infty$, to the process $Y^\kappa_t$, first, one should check the weak compactness of the family $Y^{\delta,\kappa}_t$, then prove the convergence inside the edges. This is a standard averaging principle (see, for instance, \cite{05}). To calculate the gluing conditions, consider, first, the case $\beta(x)\equiv 0$. Then the convergence can be proved similarly to Theorem 8.5.1 from \cite{05}. Then one should use the fact that the addition of the drift $\beta(x)$ leads to absolutely continuous transformation of corresponding measures in the space of trajectories (the Girsanov formula). The last step is similar to the argument in \cite{01}. 

The process $Y^\kappa_t$ is a diffusion process on a one-dimensional structure, and many of its characteristics can be calculated explicitly as solutions of corresponding linear ordinary differential equations satisfying the gluing condition and appropriate boundary conditions. In particular, one can calculate explicitly expected values of the exit times from a neighborhood of the vertex $O$ and the distribution of $Y^\kappa_t$ at the first exit time. This allows to calculate the limit $Y_t$ of the processes $Y^\kappa_t$ as $\kappa \downarrow 0$. The process $Y_t$ is deterministic inside the edges and has certain stochasticity at $O$: after reaching the vertex $O$, the trajectories of $Y_t$ without any delay go to the edges $I_1$, $I_2$, or $I_3$ with certain probabilities which are calculated explicitly.

Thus the processes $Y^{\delta,\kappa}_t$, $0 \le t \le T$, converge weakly in the space of continuous functions $[0,T] \mapsto \Gamma$ to the process $Y_t$ on $\Gamma$ as, first, $\delta \downarrow 0$ and then $\kappa \downarrow 0$. This leads to the following result:

\begin{thm}\label{thm01}
Let $X^{\delta,\kappa}_t$ be the process defined by equation \eqref{06} and
\[V_{ij}^\delta=\inf\left\{ S_{0T}^\delta(\varphi): \varphi \in C_{0T} : \varphi_0=O_i, \varphi_T=O_j, T \ge 0 \right\},\]
where $\frac{1}{\kappa}S^\delta_{0T}(\varphi)$ is the action functional for the family $X^{\delta,\kappa}_t$ as $\kappa \downarrow 0$ and $\delta>0$ is fixed. Then  $\lim_{\delta \downarrow 0} V^\delta_{ij}=:\alpha_i$ exists and
\[\alpha_i=\lim_{\delta \downarrow 0}V_{ij}^\delta =\left|\int_{H(O_i)}^O \frac{\hat{\beta}_i(z)}{\hat{\alpha}_i(z)}dz\right|, \quad i \in \{1,2,3\}. \]

Let $i, i_1, i_2 \in \{1,2,3\}$ and no two of them are equal. For a small $\rho>0$, put $\tau^{\delta,\kappa}_i=\min\{t : \min_{k=1,2}|X^{\delta,\kappa}_t-O_{i_k}|=\rho\}$. Then for any $x \in G_i$,
\[\lim_{\kappa \downarrow 0}\lim_{\delta \downarrow 0} P_x\left\{X^{\delta,\kappa}_{\tau^{\delta,\kappa}_i} \in G_{i_k}\right\}= \frac{\sqrt{\hat{a}_{i_k} \gamma_{i_k}}}{\sqrt{\hat{a}_{i_1} \gamma_{i_1}}+\sqrt{\hat{a}_{i_2} \gamma_{i_2}}},\quad k=1,2,\] 
where $\hat{a}_i$ were defined in \eqref{09}, and 
\[\left.\gamma_i=\frac{d}{dz}\int_{G_i(z)}\text{div}\beta(x)dx\right|_{z=O}=\oint_{C_i}\frac{\text{div}\beta(x)dl}{|\nabla H(x)|};\]
$C_i$ is the boundary of $G_i$.

Put $\overline{\tau}^{\delta,\kappa}=\min\{t: \min_{k=1,2,3}|X^{\delta,\kappa}_t-O_k|=\rho \}$. Then for any $x=X^{\delta,\kappa}_0\in G_4$,
\[\lim_{\kappa \downarrow 0}\lim_{\delta \downarrow 0} P_x\{ |X^{\delta,\kappa}_{\overline{\tau}^{\delta,\kappa}}-O_i|=\rho \}=\frac{\sqrt{\hat{a}_i \gamma_i}}{\sqrt{\hat{a}_1 \gamma_1}+\sqrt{\hat{a}_2 \gamma_2}+\sqrt{\hat{a}_3 \gamma_3}},\quad i=1,2,3.\]

\end{thm}

This result can be interpreted as follows. Let $\alpha_1 <\alpha_2<\alpha_3$ and $0 < \delta \ll \kappa \ll 1$. Let $\lim_{\kappa \downarrow 0}\kappa \ln T_\lambda(\kappa)=\lambda$. Then, if $\lambda>\alpha_2$, the metastable state for the time scale $T_\lambda(\kappa)$ and any initial point $X^{\delta,\kappa}_0=x  \in \RR^2$ will be $O_3$: with probability close to one, $X^{\delta,\kappa}_{T_\lambda(\kappa)}$ will be situated in a small neighborhood of $O_3$, and most of the time during the time interval $[0,T_\lambda(\kappa)]$, the trajectory $X^{\delta,\kappa}_t$ spends in this small neighborhood of $O_3$.

If $0 <\lambda<\alpha_1$, then the metastable state for an initial point $X^{\delta,\kappa}_0=x \in G_i$, $i \in \{1,2,3\}$, will be $O_i$. If $x \in G_4$, then we have metastable distribution: $X^{\delta,\kappa}_{T_\lambda(\kappa)}$ is situated in a small neighborhood of $O_i$, $i \in \{1,2,3\}$, with probability $\frac{\sqrt{\hat{a}_i \gamma_i}}{\sqrt{\hat{a}_1 \gamma_1}+\sqrt{\hat{a}_2 \gamma_2}+\sqrt{\hat{a}_3 \gamma_3}}$.

If $\alpha_1 <\lambda<\alpha_2$ and $X^{\delta,\kappa}_0=x \in G_1 \cup G_4$, then as, first, $\delta \downarrow 0$ and then $\kappa \downarrow 0$, $X^{\delta,\kappa}_{T_\lambda(\kappa)}$ is distributed between small neighborhoods of $O_2$ and $O_3$ with probabilities 
$\frac{\sqrt{\hat{a}_2 \gamma_2}}{\sqrt{\hat{a}_2 \gamma_2}+\sqrt{\hat{a}_3 \gamma_3}}$ and 
$\frac{\sqrt{\hat{a}_3 \gamma_3}}{\sqrt{\hat{a}_2 \gamma_2}+\sqrt{\hat{a}_3 \gamma_3}}$ respectively.

If $\alpha_1 <\lambda<\alpha_2$ and $x \in G_i$, $i=2,3$, then the metastable state is $O_i$.

\section{Hierarchy of Markov chains}

Consider system \eqref{03} with the vector field $b(x)$, $x \in \RR^n$, having a finite number of asymptotically stable equilibriums $K_1,...,K_l$ separated by separatrix surfaces of dimension $n-1$. Let $\Pi$ be the union of separatrix surfaces. The case of more general asymptotically stable compacts can be considered in a similar way if each of these compacts supports just one normalized invariant measure. Assume for brevity that each point $x \in \RR^n\backslash \Pi$ is attracted to one of $K_i$.

In Section 1, we introduced the action functional $\epsilon^{-1}S_{0T}(\varphi)$ for the family of processes $X^\epsilon_t$, $0\le t \le T$, defined by \eqref{03}, the function $V(x,y)$, $x,y \in \RR^n$, and numbers $V_{ij}=V(K_i,K_j)$, $i,j \in \mathcal{L}=\{1,2,...,l\}$.

For each $i \in \mathcal{L}$, consider $\alpha_i=\min_{j \in \mathcal{L}\backslash \{i\}}V_{ij}$. This minimum can be achieved at $d(i)$ points $j_1^*(i),...,j_{d(i)}^*(i) \in \mathcal{L}\backslash \{i\},$ $1 \le d(i) \le l-1$. Connect the point $i$ with each of $j_1^*(i),...,j_{d(i)}^*(i)$ by an arrow leading from $i$ to $j_k^*(i)$,  $k \in \{1,...,d(i)\}$. So now from each point $i \in \mathcal{L}$, one or several arrows start. Denote by $\mathcal{E}(i)$ the maximal subset of $\mathcal{L}$ such that: (1) $i \in \mathcal{E}(i)$, (2) for any $k \in \mathcal{E}(i)$, $k \neq i$, a sequence of arrows exists leading from $i$ to $k$ as well as a sequence leading from $k$ to $i$. Note that if $\widetilde{\mathcal{E}}(i)$ and $\widetilde{\widetilde{\mathcal{E}}}(i)$ are subsets of $\mathcal{L}$ satisfying conditions (1) and (2) then  $\widetilde{\mathcal{E}}(i) \cup \widetilde{\widetilde{\mathcal{E}}}(i)$ also satisfies these conditions; for $k \in  \mathcal{E}(i)$, $\mathcal{E}(k)=\mathcal{E}(i)$. So that the maximal set $\mathcal{E}(i)$ satisfying conditions (1) and (2) exists, and sets $\mathcal{E}(i)$ form a partition of $\mathcal{L}: \mathcal{L}=\mathcal{E}_1\cup \mathcal{E}_2 \cup...\cup \mathcal{E}_r$, $\mathcal{E}_j \cap \mathcal{E}_k=\emptyset$ if $j \neq k$. A set $\mathcal{E}_i$ can consist of one point or of any number less than or equal to $l$; $1 \le r \le l-1$.

On every set $\mathcal{E}_k$, $k \in \{1,...,r\}$, define a Markov chain: put $p_{ij}=p_{ij}^{k,\epsilon}=\exp\{-\alpha_i \epsilon^{-1}\}$ if there is an arrow leading from $i \in \mathcal{E}_k$ to $j \in \mathcal{E}_k$; if there is no such an arrow and $i \neq j$, put $p_{ij}=0$. Define the diagonal elements of the transition matrix by equalities $p_{ii}=1-d(i)\exp\{-\alpha_i \epsilon^{-1}\}$, where $d(i)$ is the number of arrows starting at $i \in \mathcal{E}_k$. This Markov chain on the phase space $\mathcal{E}_k$ will be also called $\mathcal{E}_k$.

It is easy to derive from the definition of $\mathcal{E}_k$ that a number $N$ exists such that the trajectory of the chain $\mathcal{E}_k$ can go from any state $i \in \mathcal{E}_k$ to any $j \in \mathcal{E}_k$ in $N=N(\mathcal{E}_k)$ steps with a positive probability. This implies that the chain $\mathcal{E}_k$ has a unique stationary distribution ${\bf q}={\bf q}^\epsilon(\mathcal{E}_k)=(q^{k,\epsilon}_1,...,q^{k,\epsilon}_{n(\mathcal{E}_k)})$, where $n(\mathcal{E}_k)$ is the number of states in $\mathcal{E}_k$.

For each $\mathcal{E}_k$, define the numbers 
\begin{align}\label{10}
&r(\mathcal{E}_k)=\max_{i \in \mathcal{E}_k} \min_{j \in \mathcal{L}\backslash \{i\}} V_{ij}=\max_{i \in\mathcal{E}_k} \alpha_i,\\
&\notag m_i(\mathcal{E}_k)=\alpha_i-r(\mathcal{E}_k) \text{ for } i \in \mathcal{E}_k,\\
&\notag e(\mathcal{E}_k)=\min_{i \in \mathcal{E}_k, j \notin \mathcal{E}_k}(r(\mathcal{E}_k)+V_{ij}-\alpha_i).
\end{align}

Let $(i_1^*, j_1^*),...,(i_{a(k)}^*,j_{a(k)}^*)$ be the points of $\mathcal{E}_k \times \{\mathcal{L} \backslash \mathcal{E}_k \}$ where the last minimum is achieved; $\{i_1^*,...,i_{a(k)}^* \}=I(\mathcal{E}_k) \subseteq \mathcal{E}_k$, $\{j_1^*,...,j_{a(k)}^* \}=J(\mathcal{E}_k) \subseteq \mathcal{L} \backslash \mathcal{E}_k$.

\begin{lem}\label{lem01}\
\begin{enumerate}[(1)]
\item Let ${\bf q}^\epsilon(\mathcal{E}_k)=(q_1^{k,\epsilon},...,q_{n(\mathcal{E}_k)}^{k,\epsilon})$ be the stationary distribution of the chain $\mathcal{E}_k$. Then 
\[\lim_{\epsilon \downarrow 0} \epsilon \ln q_i^{k,\epsilon}=\alpha_i-r(\mathcal{E}_k), \quad 1 \le i \le n(\mathcal{E}_k), \quad 1 \le k \le r.\]
\item Let $\left(p_{ij}^{(n),k,\epsilon}\right)$ be the transition matrix of the chain $\mathcal{E}_k$ in $n$ steps: $\left(p_{ij}^{(n),k,\epsilon}\right)=\left(p_{ij}^{k,\epsilon}\right)^n.$ Let a function $T_\lambda(\epsilon)$ satisfy the equality $\lim_{\epsilon \downarrow 0} \epsilon \ln T_\lambda(\epsilon)=\lambda>0$ and $[T_\lambda(\epsilon)]$ be the integer part of $T_\lambda(\epsilon)$. Then
\[\lim_{\epsilon \downarrow 0} \left|p_{ij}^{([T_\lambda(\epsilon)]),k,\epsilon}-q_j^{k,\epsilon}\right|=0\]
for $i,j \in \{1,...,n(\mathcal{E}_k)\}$, $1 \le k \le r$, if $\lambda > r(\mathcal{E}_k)$.
\item Let $Z^\epsilon_n$ be the Markov chain in the phase space $\mathcal{L}$ with transition probabilities $p_{ij}^\epsilon =\exp \{-\epsilon^{-1}V_{ij}\}$ if $i \neq j$ and $p_{ii}^\epsilon=1-\sum_{j:j\neq i} p_{ij}^\epsilon$. Let $\tau^\epsilon(\mathcal{E}_k)$ be the first time when $Z_n^\epsilon$ leaves $\mathcal{E}_k$: $\tau^\epsilon(\mathcal{E}_k)=\min\{n: Z_n^\epsilon \notin \mathcal{E}_k \}$. Then, for any $\delta >0$, 
\[\lim_{\epsilon \downarrow 0} P_i\left\{ \exp\left\{\frac{1}{\epsilon}(e(\mathcal{E}_k)-\delta)\right\} \le \tau^\epsilon(\mathcal{E}_k) \le \exp\left\{\frac{1}{\epsilon}(e(\mathcal{E}_k)+\delta)\right\} \right\} =1\]
for each initial point $Z_0^\epsilon=i \in \mathcal{E}_k$.
\item For each initial point $Z_0^\epsilon=i \in \mathcal{E}_k$, 
\[\lim_{\epsilon \downarrow 0}P_i\left\{Z^\epsilon_{\tau^\epsilon(\mathcal{E}_k)} \in J(\mathcal{E}_k)\right\}=1, \quad \lim_{\epsilon \downarrow 0}P_i\left\{Z^\epsilon_{\tau^\epsilon(\mathcal{E}_k)-1} \in I(\mathcal{E}_k)\right\}=1.\]
\end{enumerate}
\end{lem} 

We will give a sketch of the proof of this lemma.

\begin{proof}[Sketch of proof]\
\begin{enumerate}[(1)]
\item A system of arrows $g$ connecting some points of a finite set $\mathcal{E}$ is called an $i$-graph over $\mathcal{E}$, $i \in \mathcal{E}$, if one arrow starts from each $j \in \mathcal{E}\backslash \{i\}$ and a sequence of arrows leading from $j \in \mathcal{E}\backslash \{i\}$ to $i$ exists for each $j \in \mathcal{E}\backslash \{i\}$. Denote $G_j(\mathcal{E}_k)$ the set of all $j$-graphs over $\mathcal{E}_k$ consisting of arrows introduced for the set $\mathcal{E}_k$. It follows from Lemma 6.3.1 of \cite{05} that 
\begin{equation}\label{11}
\lim_{\epsilon \downarrow 0} \epsilon \ln q^{k,\epsilon}_j=\min_{g \in G_j(\mathcal{E}_k)} \sum_{(m\rightarrow n)\in g}V_{mn}-\max_{m \in \mathcal{E}_k} \min_{n \in \mathcal{L}\backslash \{m\}} V_{mn}.
\end{equation}
Taking into account that $p_{ij}^{k,\epsilon}=\exp\{-\alpha_i\epsilon^{-1}\}$, $i \neq j$, and $r(\mathcal{E}_k)=\max_{i \in \mathcal{E}_k}\alpha_i$, it is easy to see that  $\sum_{(m\rightarrow n)\in g}V_{mn}= \sum_{i:i \neq j}\alpha_i$ for each $g \in G_j(\mathcal{E}_k)$. Thus \eqref{11} implies the first statement of Lemma \ref{lem01}.

\item Let $\delta=\lambda - r (\mathcal{E}_k)>0$ and $\overline{j} \in \mathcal{E}_k$ be such that $\alpha_{\overline{j}}=r(\mathcal{E}_k)$. It follows from the definition of the chain $\mathcal{E}_k$ that a trajectory starting from any $i \in \mathcal{E}_k$ reaches $\overline{j}$ before time $t_\epsilon=\left[\exp\left\{\frac{r(\mathcal{E}_k)+\frac{\delta}{4}}{\epsilon} \right\}\right]$ with probability greater than $\exp\left\{-\frac{\delta}{8\epsilon}\right\}$ as $\epsilon \downarrow 0$. This implies that starting from any $i \in \mathcal{E}_k$, the trajectory of $\mathcal{E}_k$ will be at $\overline{j}$ at time $t_\epsilon$ with probability $\delta_\epsilon \ge \exp\left\{-\frac{\delta}{4\epsilon}\right\}$ if $\epsilon>0$ is small enough. So that the modified D\"{o}blin condition is satisfied. Using a standard estimate (see \cite{02} Section 5.2):
\begin{align}\label{12}
\left|p_{ij}^{([T_\lambda(\epsilon)]),k,\epsilon}-q_j^{k,\epsilon}\right|&\le (1-\delta_\epsilon)^{\frac{[T_\lambda(\epsilon)]}{t_\epsilon}-1} \sim \exp\left\{-\delta_\epsilon\frac{[T_\lambda(\epsilon)]}{t_\epsilon}+\delta_\epsilon\right\}\\
&\notag\le\exp\left\{-\exp \left\{-\frac{\delta}{4\epsilon}+\frac{r(\mathcal{E}_k)+(\lambda-r(\mathcal{E}_k))-r(\mathcal{E}_k)}{\epsilon}-\frac{\delta}{4\epsilon} \right\}+\frac{\delta}{\epsilon} \right\}.
\end{align}
Choosing $\delta <\lambda - r (\mathcal{E}_k)$, we conclude from \eqref{12}, for $\epsilon>0$ small enough,
\begin{equation}\label{13}
\left|p_{ij}^{([T_\lambda(\epsilon)]),k,\epsilon}-q_j^{k,\epsilon}\right|\le \exp\left\{-\exp\left\{ \frac{\delta}{4\epsilon}\right\} \right\}.
\end{equation}

\item First, taking into account that $\max_{i \in \mathcal{E}_k, m \notin \mathcal{E}_k}(V_{im}-\alpha_i)=h>0$, one can show that for $\delta>0$ small enough, trajectory $Z^\epsilon_n$ of the chain $\mathcal{E}_k$ starting from $i \in \mathcal{E}_k$ will not leave $\mathcal{E}_k$ before time $\exp\left\{\frac{1}{\epsilon}(r(\mathcal{E}_k)+\delta)\right\}$ with probability close to 1 if $\epsilon \downarrow 0$. On the other hand, according to part (2) of this lemma, distribution of $Z^\epsilon_n$ at time $n=\left[\exp\left\{\frac{1}{\epsilon}(r(\mathcal{E}_k)+\delta)\right\}\right]$ will be close to the stationary distribution described in part (1). So that the  distribution of $\tau^\epsilon(\mathcal{E}_k)$ starting from any $i \in \mathcal{E}_k$ is close to the distribution of $\tau^\epsilon(\mathcal{E}_k)$ for trajectories with the stationary initial distribution as $\epsilon \ll 1$. This allows to prove statement (3) similarly to the proof of Theorem 6.6.2 from \cite{05}.
\end{enumerate}

The last statement of the lemma can be proved using result of part (3) and arguments from Section 5 of \cite{05}.

\end{proof}

The number $r(\mathcal{E}_k)$ defined in \eqref{10}, we call the convergence rate for the chain $\mathcal{E}_k$; numbers $m_i(\mathcal{E}_k)$ are called invariant measure rates; numbers $e(\mathcal{E}_k)$ are called exit rates. The set $I(\mathcal{E}_k)$ is called exit set from $\mathcal{E}_k$, $J(\mathcal{E}_k)$ is the set of positions after exiting $\mathcal{E}_k$. The set $\{i \in \mathcal{E}_k : \alpha_i=r(\mathcal{E}_k)\}$ is called main subset of $\mathcal{E}_k$. Lemma \ref{lem01} gives motivations for these names.\\

We call the Markov chains $\mathcal{E}_1,...,\mathcal{E}_r$ chains of rank 1. A specific property of these chains consists of the fact that their transition probabilities $p_{ij}^\epsilon$ in $i$-th row are equal either to $\exp\{-\alpha_i\epsilon^{-1}\}>0$ or to 0. We will define now chains of rank 2 and higher. If necessary, we will provide the notations related to a chain of rank $m$ with an index $m$; for instance we will write $\mathcal{E}_k^{(1)}$ and $\alpha_k^{(1)}$ for $\mathcal{E}_k$ and $\alpha_k $.

We say that a chain $\mathcal{E}_j^{(1)}$ of rank 1 (1-chain) follows after 1-chain $\mathcal{E}_i^{(1)}$ if $\mathcal{E}_j^{(1)}$ contains, at least, one element of $J(\mathcal{E}_i^{(1)})$: $J(\mathcal{E}_i^{(1)}) \cup \mathcal{E}_j^{(1)} \neq \emptyset$. Introduce a system of arrows in the set $\mathcal{E}^{(1)}=(\mathcal{E}_1^{(1)},...,\mathcal{E}_{r^{(1)}}^{(1)})$. We draw an arrow from $\mathcal{E}_i^{(1)}$ to $\mathcal{E}_j^{(1)}$ if $\mathcal{E}_j^{(1)}$ follows after $\mathcal{E}_i^{(1)}$. So, from each $\mathcal{E}_k^{(1)}$ starts at least one arrow.

For each $\mathcal{E}_i^{(1)} \in \mathcal{E}^{(1)}$, consider a maximal subset $\mathcal{E}^{(2)}(\mathcal{E}_i^{(1)})$ of $\mathcal{E}^{(1)}$ containing $\mathcal{E}_i^{(1)}$  and such that any two points of $\mathcal{E}^{(2)}(\mathcal{E}_i^{(1)})$ are connected by a sequence of arrows leading from one point to another. In this way, we get a partition of $\mathcal{E}^{(1)}$: $\mathcal{E}^{(1)}=\cup_{k=1}^{r^{(2)}} \mathcal{E}_k^{(2)}$, $\mathcal{E}_i^{(2)} \cap \mathcal{E}_j^{(2)}=\emptyset$ if $i \neq j$, $i,j \in \{1,...,r^{(2)} \}$, where $r^{(2)}$ is the number of elements in the partition. 

Define now a Markov chain on each $\mathcal{E}^{(2)}_k$: let $p_{ij}^{(2)}=p^{(2),k}_{\mathcal{E}_i^{(1)},\mathcal{E}_j^{(1)}}=\exp\{-\epsilon^{-1} e(\mathcal{E}_i^{(1)}) \}$ if $\mathcal{E}_i^{(1)},\mathcal{E}_j^{(1)} \in \mathcal{E}_k^{(2)}$, $i \neq j$, and an arrow leading from $\mathcal{E}_i^{(1)}$ to $\mathcal{E}_j^{(1)}$ exists; if there is no such an arrow, put $p_{ij}^{(2)}=0$. The probabilities $p_{ii}^{(2)}$ define so that sum of elements of each row of $\left(p_{ij}^{(2)}\right)$ is equal to 1. This chain on the state space $\mathcal{E}_k^{(2)}$ is also called $\mathcal{E}_k^{(2)}$. The chains $\mathcal{E}_1^{(2)},...,\mathcal{E}_{r^{(2)}}^{(2)}$ are called rank 2 chains (2-chains). For each 2-chain $\mathcal{E}_k^{(2)}$ define the convergence rate $r(\mathcal{E}_{k}^{(2)})$, the invariant distribution rates $m_{\mathcal{E}_{i}^{(1)}}(\mathcal{E}_{k}^{(2)})=m_i(\mathcal{E}_{k}^{(2)})$ for $\mathcal{E}_{i}^{(1)} \in \mathcal{E}_{k}^{(2)}$, and the exit rate $e(\mathcal{E}_{k}^{(2)})$ as such numbers were defined for 1-chains in \eqref{10}.

In the similar way as we did it  for 1-chains, the sets $I( \mathcal{E}_k^{(2)})\subseteq \mathcal{E}_k^{(2)} $ and $J( \mathcal{E}_k^{(2)})\not\subseteq \mathcal{E}_k^{(2)} $ are defined. 

Now we can define 3-chains, 4-chains, and so on by induction. Note that the number $n_k$ of $k$-chains satisfies the condition $n_k \le n_{k-1}-1$.  So that rank cannot be larger than $l-1$ where $l$ is the number of asymptotically stable attractors. The chain of the highest rank covers all $l$ points. See an example in the next Section. The hierarchy of Markov chains is a generalization of the hierarchy of cycles introduced in \cite{03} (see also \cite{04,05}).

\section{Metastable sets}

As was shown in \cite{03,04} (see also \cite{05}), in the generic case (when there is no rough symmetry), the long-time behavior of the process $X^\epsilon_t$ defined by equation \eqref{03} can be described by a hierarchy of cycles. It is convenient to consider the attractors $K_1$,...,$K_l$ (Sometimes, we denote $K_i$ just by $i$.) as cycles of rank 0. The states of a $(k+1)$-cycle, $k>0$, are some cycles of rank $k$ ($k$-cycles). One arrow starts at each $k$-cycle $\mathcal{C}$ included in a $(k+1)$-cycle $\mathcal{C}^{(k+1)}$ and leads to another $k$-cycle $\widetilde{\mathcal{C}} \in \mathcal{C}^{(k+1)}$. Starting at $\mathcal{C} \in \mathcal{C}^{(k+1)}$ and moving along the arrows one comes back to $\mathcal{C}$.

Let the initial point $X^\epsilon_0=x \in \RR^n$ be attracted to $K_i$. In the generic case for each $i \in \mathcal{L}$, one can consider a sequence of cycles of growing rank: $\{K_i\}=\mathcal{C}^{(0)}(i) \subseteq \mathcal{C}^{(1)}(i) \subseteq ... \subseteq \mathcal{C}^{(m)}(i)$ where $m=m(i)<l$.

The convergence (rotation) rate $r(\mathcal{C}^{(k)}(i) )$, the invariant measure rates $m_j(\mathcal{C}^{(k)}(i) )$, $j \in \mathcal{C}^{(k)}(i)$, and the exit rate $e(\mathcal{C}^{(k)}(i) )$ are introduced for each cycle $\mathcal{C}^{(k)}(i) $ as well as the main state $M(\mathcal{C}^{(k)}(i) )$. All these characteristics are defined by the numbers $V_{ij}$. The exit rates $e(\mathcal{C}^{(0)}(i) ),...,e(\mathcal{C}^{(m)}(i) )$ form a non-decreasing sequence, $e(\mathcal{C}^{(m)}(i) )=\infty$, $r(\mathcal{C}^{(0)}(i) )=0$. It was shown in mentioned above papers there that if $\lim_{\epsilon \downarrow 0} \epsilon \ln T_\lambda(\epsilon)=\lambda>0$, $e(\mathcal{C}^{(k-1)}(i) )<\lambda<e(\mathcal{C}^{(k)}(i) )$, and $r(\mathcal{C}^{(k)}(i))<\lambda$, then the trajectory $X^\epsilon_t$, $X^\epsilon_0=x$, during time interval $[0,T_\lambda(\epsilon)]$  spends most of the time near $M(\mathcal{C}^{(k)}(i) )$ and $X^\epsilon_{T_\lambda(\epsilon)}$ is close to $M(\mathcal{C}^{(k)}(i) )$ with probability close to 1 as $\epsilon \downarrow 0$. In this case, $M(\mathcal{C}^{(k)}(i) )$ is called the metastable state for a given initial point $x$ and a time scale $T_\lambda(\epsilon)$. If $e(\mathcal{C}^{(k-1)}(i) )<\lambda<e(\mathcal{C}^{(k)}(i) )$ but $\lambda <r(\mathcal{C}^{(k)}(i) )$, then a bit more sophisticated procedure for calculation of the metastable state was suggested. The Markov chains introduced in Section 3 in the case of generic systems have just one non-zero non-diagonal element in each row of their transition matrix. These chains are closely related to the cycles.

If system \eqref{03} has a rough symmetry, then to find the long-time behavior, one should introduce the hierarchy of Markov chains related to this system. For each $K_i$, consider a sequence of Markov chains of growing rank: 
\begin{equation}\label{14}
\{i\}=\mathcal{E}^{(0)}(i) \subseteq \mathcal{E}^{(1)}(i) \subseteq ...\mathcal{E}^{(m)}(i).
\end{equation}
Here $m$ is such that $\mathcal{E}^{(m)}(i)$ contains all $j \in \mathcal{L}$; $m=m(i)$. For each $\mathcal{E}^{(k)}(i)$, consider the numbers $r(\mathcal{E}^{(k)}(i))=r_k(i)$, $m_j(\mathcal{E}^{(k)}(i))=m_{k,j}(i)$, $e(\mathcal{E}^{(k)}(i))=e_k(i)$. Let $M(\mathcal{E}^{(k)}(i))=\{j \in \mathcal{E}^{(k)}(i): m_{k,j}(i)=0\}.$ The numbers $0, e_1(i), ..., e_m(i)=\infty$  form a non-decreasing sequence.

\begin{thm}\label{thm02}
Let the initial point $x$ of the process $X^\epsilon_t$ defined by \eqref{03} be attracted to $K_i$. Let $\lim_{\epsilon \downarrow 0} \epsilon \ln T_\lambda(\epsilon)=\lambda>0$. Assume that $e_{k-1}(i)<\lambda<e_k(i)$ and $r_k(i)<\lambda$. Then, for any $\delta>0$, 
\[\lim_{\epsilon \downarrow 0} P_x\left\{X^\epsilon_{T_\lambda(\epsilon)} \in U_\delta\left(M(\mathcal{E}^{(k)}(i))\right)\right\}=1,\]
where $U_\delta\left(M(\mathcal{E}^{(k)}(i))\right)=\{x \in \RR^n : \min_{j \in M(\mathcal{E}^{(k)}(i)) }|x-K_j|<\delta\}$ is the $\delta$-neighborhood of the set $M(\mathcal{E}^{(k)}(i))$ in $\RR^n$.
\end{thm}

\begin{proof}
Taking into account Lemma~\ref{lem01}, the proof of Theorem~\ref{thm02} can be carried out using the same arguments as in Section 6.5 of \cite{05}.
\end{proof}

For any $\lambda>0$, except a finite set, one can find $k$ such that $\lambda \in \left( e(\mathcal{E}^{(k-1)}(i)),e(\mathcal{E}^{(k)}(i))\right)$. But the second assumption of Theorem~\ref{thm02} that $r(\mathcal{E}^{(k)}(i))<\lambda$ is, of course, satisfied not always. We will show now, how the general case can be reduced to the situation considered in Theorem~\ref{thm02}. 

Let, first, the numbers $V_{ij}$ be such that the rank 1 chain $\mathcal{E}^{(1)}(i)$ contains all points of $\mathcal{L}$. In this case $\mathcal{E}^{(1)}(i)=\mathcal{E}$ is the same for all $i \in \mathcal{L}$, and the maximal rank in sequence \eqref{14} $m(i)=m=1.$ Assume that $r(\mathcal{E})>\lambda$. Consider  the arrows in the set $\mathcal{E}$ defined by the numbers $V_{ij}$. Let $\mathcal{A}(i,\lambda)=\{j \in \mathcal{E} : \text{ a sequence of arrows } (m \rightarrow n)\text{ leading from } i \text{ to } j \text{ exists such that } V_{mn}<\lambda \}$. The set $\mathcal{E}\backslash \mathcal{A}(i,\lambda)$ is not empty because of our assumption $r(\mathcal{E})>\lambda$. Put $M(i,\lambda)=\{k \in \mathcal{A}(i,\lambda): \min_{j \in \mathcal{L}, j \neq k} V_{kj}>\lambda\}$. The set $M(i,\lambda)$ is also not empty since if it is empty, $\mathcal{E}$ will not be a rank 1 chain.

Using standard large deviation estimates, one can check that the process $X^\epsilon_t$ (with $X^\epsilon_0=x$ attracted to $K_i$) hits any neighborhood of the set $M(i,\lambda) \in \RR^n$ (we identify the numbers $i$ and the points $K_i$) before time $T_{\lambda_1}(\epsilon)$, $\lim_{\epsilon \downarrow 0} \epsilon \ln T_{\lambda_1}(\epsilon)=\lambda_1<\lambda$, with probability close to 1 as $\epsilon \downarrow 0$. Moreover, $X^\epsilon_t$ will stay in the union of basins of attractors $K_j$, $j \in M(i,\lambda)$ till $T_\lambda(\epsilon)$ since for a switch to the basin of another attractor, one needs time larger than $\exp \left\{\frac{1}{\epsilon}\min_{k \in M(i,\lambda), j \neq k} V_{kj}\right\}$. We summarize this in the following lemma.

\begin{lem}\label{lem02}

Let the numbers $V_{ij}$ be such that $\mathcal{E}_1(i)=\mathcal{E}=\mathcal{L}$; the set $\mathcal{E}$ is provided with corresponding system of arrows and with numbers $\alpha_i=\min_{k,k\neq i}V_{ik}$. Let $r(\mathcal{E})=\max_{j \in \mathcal{E}} \alpha_j$ be the convergence rate for the chain $\mathcal{E}$. Let $\lim_{\epsilon \downarrow 0} \epsilon \ln T_\lambda(\epsilon)=\lambda>0$. Put $M(i,\lambda)=\{j \in \mathcal{E} : \text{ there exists a sequence of arrows } (i=i_0 \rightarrow i_1), (i_1 \rightarrow i_2),...,(i_{d-1}\rightarrow j)\text{ leading from } i \in \mathcal{E} \text{ to } j \text{ with } \alpha_{i_k}<\lambda \text{ for } k=0,1,...,d-1 \text{ and } \alpha_j> \lambda \}.$ If $\alpha_i>\lambda$, put $M(i,\lambda)=\{ i \}$. Assume that $r(\mathcal{E})>\lambda$. Then $M(i,\lambda)$ is not empty. Denote by $U_\delta(M(i,\lambda))$ the $\delta$-neighborhood of $M(i,\lambda)$ in $\RR^n$. Then $\lim_{\epsilon \downarrow 0} P_x\{X^\epsilon_{T_\lambda(\epsilon)} \in U_\delta(M(i,\lambda)) \}=1$ for any $\delta>0$.

\end{lem}

An algorithm for the construction of the metastable set in general case can be described by induction in the rank. Lemma~\ref{lem01} gives the answer in the case when the highest rank is 1 and $r(\mathcal{E})>\lambda$. Suppose we can construct the metastable set in the case when the rank is less than $k$. Let $e(\mathcal{E}^{(k-1)}(i))<\lambda<e(\mathcal{E}^{(k )}(i))$. If $r(\mathcal{E})<\lambda$, the answer is given by Theorem~\ref{thm02}, so that let $r(\mathcal{E})>\lambda$. The states of $k$-chain $\mathcal{E}^{(k)}(i)$ are $(k-1)$-chains $\mathcal{E}^{(k-1)}_1$, $\mathcal{E}^{(k-1)}_2$,..., $\mathcal{E}^{(k-1)}_d$. Assume that $i \in \mathcal{E}^{(k-1)}_1$. Similar to the construction of the set $M(i,\lambda)=M^1(i,\lambda)$ in the case of rank 1 chains, one  can find $(k-1)$ chains in $\mathcal{E}^{(k)}(i)$ which can be entered before and cannot be left at time $\exp\{\frac{\lambda}{\epsilon} \}$. Let $M^k(1,\lambda)$ be the collection of all such $(k-1)$-chains. One can also describe the entry points for each of these $(k-1)$-chains. Then the metastable set for initial point $x$ attracted to $K_i$ in time scale $T_\lambda(\epsilon)$, $\lim_{\epsilon \downarrow 0} \epsilon \ln T_\lambda(\epsilon)=\lambda$, can be described as the union of metastable sets for given $x$ and $\lambda$ inside each of the $(k-1)$-cycles from $M^k(x,\lambda)$. The latter can be described due to the induction assumption.\\

Consider some examples. First, consider briefly the system studied in Section 2. There are three asymptotically stable states in this system and the unique rank 1 chain $\mathcal{E}_1^{(1)}$ contains all three points. Corresponding transition matrix and system of arrows are shown in Figure 3.\\\\

\begin{figure}[h]
\centering
\begin{picture}(330, 140)
    \put(0,0){\includegraphics[scale=0.5]{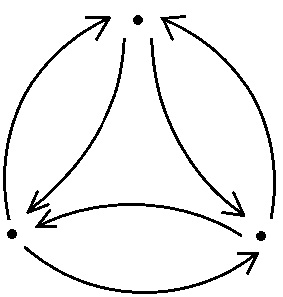}}
	\put(-5,75){$\alpha_1$}	
	\put(50,-8){$\alpha_1$}	
	\put(37,55){$\alpha_2$}		
	\put(55,55){$\alpha_2$}	
	\put(50,25){$\alpha_3$}
	\put(100,75){$\alpha_3$}
			
	\put(-5,15){1}	
	\put(50,108){2}		
	\put(102,15){3}	
	
	\put(170,50){\Large $\vb{1-2e^{-\frac{\alpha_1}{\epsilon}} & e^{-\frac{\alpha_1}{\epsilon}} &e^{-\frac{\alpha_1}{\epsilon}}}{e^{-\frac{\alpha_2}{\epsilon}}&1-2e^{-\frac{\alpha_2}{\epsilon}}&e^{-\frac{\alpha_2}{\epsilon}}}{e^{-\frac{\alpha_3}{\epsilon}}&e^{-\frac{\alpha_3}{\epsilon}}&1-2e^{-\frac{\alpha_3}{\epsilon}}}$}
\end{picture}
\caption*{Figure 3}
\end{figure}

Taking into account that $\alpha_1<\alpha_2<\alpha_3$, we have: $r(\mathcal{E}_1^{(1)})=\alpha_3$, $m_i(\mathcal{E}_1^{(1)})=\alpha_i-\alpha_3$ for $i \in \{1,2,3\}$, $e(\mathcal{E}_1^{(1)})=\infty$, and  $M(\mathcal{E}_1^{(1)})=\{3\}$. Metastable sets for this system were described in Section 2. 

Note that in the case of the system shown in Figure 1, the non-perturbed trajectories started at points $x \in G_4$ are not attracted to any stable equilibrium. But the perturbed trajectories approach one of asymptotically stable equilibriums relatively fast (in a time of order $\ln \epsilon^{-1}$ as $\epsilon \downarrow 0$). The last statement of Theorem~\ref{thm01} gives the distribution between the attractors which should be considered as an initial distribution for exponentially long time evolution.

Consider another example. Let system~\eqref{01} have $l=5$ asymptotically stable equilibriums (Figure $4^a$). If there is an arrow $(i \rightarrow j)$ in Figure $4^a$, then $V_{ij}=\min_{k:k\neq i}V_{ik}=:\alpha_i$. Assume that $\alpha_1<\alpha_2<...<\alpha_5$. If no arrow leads from $i$ to $j$, $i \neq j$, then $V_{ij}>\alpha_5$. 

\begin{figure}[h]
\centering
\begin{picture}(400, 130)
    \put(0,0){\includegraphics[scale=0.6]{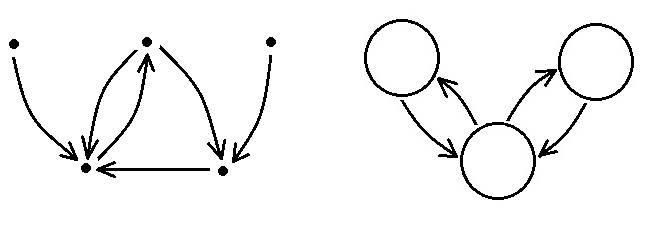}}
	\put(2,45){$\alpha_5$}	
	\put(34,62){$\alpha_1$}	
	\put(93,62){$\alpha_1$}			
	\put(63,47){$\alpha_3$}	
	\put(68,17){$\alpha_2$}
	\put(123,52){$\alpha_4$}
			
	\put(66,88){1}	
	\put(102,14){2}			
	\put(38,14){3}	
	\put(120,87){4}	
	\put(7,87){5}	

	\put(211,60){$e_1$}	
	\put(227,66){$e_1$}	
	\put(260,39){$\alpha_4$}	
	\put(178,39){$\alpha_5$}

	\put(174,74){$\mathcal{E}^{(1)}_3$}	
	\put(217,25){$\mathcal{E}^{(1)}_1$}	
	\put(261,72){$\mathcal{E}^{(1)}_2$}

	\put(295,60){$\displaystyle e_1=\alpha_3+\min_{\begin{smallmatrix}
	 j \in \{1,2,3\}, \\
	k \in \{4,5\}
	\end{smallmatrix} } (V_{jk}-\alpha_j)$}	
	
	\put(70,0){(a)}	
	\put(300,0){(b)}	

\end{picture}
\caption*{Figure 4}
\end{figure}

There are three chains of rank 1 in Figure $4^a$: $\mathcal{E}_1^{(1)}=\{1,2,3\}$, and two chains consisting just of one state, $\mathcal{E}^{(1)}_2=\{4\}$ and $\mathcal{E}^{(1)}_3=\{5\}$. It is easy to calculate that $r(\mathcal{E}_1^{(1)})=\alpha_3$, $m_i(\mathcal{E}_1^{(1)})=\alpha_i-\alpha_3$ for $i \in \{1,2,3\}$, $M(\mathcal{E}_1^{(1)})=\{j: m_j(\mathcal{E}_1^{(1)})=0\}=\{3\}$, $e(\mathcal{E}_1^{(1)})=\alpha_3+\min_{j \in \{1,2,3\}, k \in \{4,5\}}(V_{jk}-\alpha_j)=:e_1$; $r(\mathcal{E}_2^{(1)})=r(\mathcal{E}_3^{(1)})=0$, $m_4(\mathcal{E}_2^{(1)})=m_5(\mathcal{E}_3^{(1)})=0$, $M(\mathcal{E}_2^{(1)})=\{4\}$, $M(\mathcal{E}_5^{(1)})=\{5\}$, $e(\mathcal{E}_2^{(1)})=\alpha_4$, $e(\mathcal{E}_3^{(1)})=\alpha_5$.

The 1-chains $\mathcal{E}_1^{(1)},\mathcal{E}_2^{(1)},\mathcal{E}_3^{(1)}$ form a rank 2 chain $\mathcal{E}_1^{(2)}=(\mathcal{E}_1^{(1)},\mathcal{E}_2^{(1)},\mathcal{E}_3^{(1)})$ (Figure $4^b$). Letters attached to the arrows on Figure $4^b$ define corresponding transition probabilities in the chain $\mathcal{E}_1^{(2)}$: $p^\epsilon_{\mathcal{E}_1^{(1)}\mathcal{E}_3^{(1)}}=p^\epsilon_{\mathcal{E}_1^{(1)}\mathcal{E}_2^{(1)}}=\exp\{-\epsilon^{-1}e_1\}$, $p^\epsilon_{\mathcal{E}_2^{(1)}\mathcal{E}_1^{(1)}}=\exp\{-\epsilon^{-1}\alpha_4\}$, $p^\epsilon_{\mathcal{E}_3^{(1)}\mathcal{E}_1^{(1)}}=\exp\{-\epsilon^{-1}\alpha_5\}$.

Assume that the initial point $X^\epsilon_0=x$ of process \eqref{03} is attracted to $K_1$. Then $X^\epsilon_t$, first, will approach $K_1$ with probability close to 1 as $\epsilon \downarrow 0$ and stay in the basin $D_1$ of $K_1$ till time $T_\lambda(\epsilon)$, $\lim_{\epsilon \downarrow 0} \epsilon \ln T_\lambda(\epsilon)=\lambda$, if $\lambda <\alpha_1$. Most of the time, $X^\epsilon_t$ will stay in a small neighborhood of $K_1$ visiting time to time other parts of $D_1$. If $\alpha_1<\lambda<\alpha_2$, $X^\epsilon_{T_\lambda(\epsilon)}$ belongs to the union of basins $D_2$ and $D_3$ of the attractors $K_2$ and $K_3$. To calculate the distribution between $D_2$ and $D_3$, one should know the pre-exponential factor in transition probabilities asymptotics or use an approach similar to one considered in Section 2.

If $\alpha_2<\lambda<e_1$, the trajectory $X^\epsilon_t$ will reach $D_3$ before time $T_\lambda(\epsilon)$ and stay in $D_3$ till time $T_\lambda(\epsilon)$ with probability close to 1 as $\epsilon \downarrow 0$. So that, if $\alpha_2<\lambda<e_1$ and $X^\epsilon_0=x \in D_1$, $K_3$ will be the metastable state.

If $\lambda>e_1$, trajectory $X^\epsilon_t$ leaves $D_1 \cup D_2 \cup D_3$ before time $T_\lambda(\epsilon)$ as $\epsilon \downarrow 0$. Where $X^\epsilon_{T_\lambda(\epsilon)}$ is situated depends on relation between $e_1$ and $\alpha_4$: if $e_1<\lambda<\alpha_4$, $X^\epsilon_{T_\lambda(\epsilon)}$ belongs to the union $D_4 \cup D_5$ of basins of $K_4$ and $K_5$; if $\alpha_4<e_1<\lambda$, $X^\epsilon_{T_\lambda(\epsilon)}$ is situated near $K_5$ with probability close to 1 as $\epsilon \downarrow 0$.\\

Finally, one should note that the long-time behavior of systems with other types of noise can be studied in a similar way.

\bibliographystyle{plain}
\bibliography{paper}

\begin{thebibliography}{1}

\bibitem{01}
M.~Brin and M.~Freidlin.
\newblock {On stochastic behavior of perturbed Hamiltonian systems}.
\newblock {\em Ergodic Theory and Dynamical Systems}, 20(1):55--76, 2000.

\bibitem{02}
J.~L. Doob.
\newblock {\em {Stochastic processes}}.
\newblock Wiley, 1953.

\bibitem{03}
M.~Freidlin.
\newblock {Sublimiting distributions and stabilization of solutions of
  parabolic equations with a small parameter}.
\newblock {\em Soviet Math. Dokl.}, 18(4):1114--1118, 1977.

\bibitem{04}
M.~Freidlin.
\newblock {Quasi-deterministic approximation, metastability and stochastic
  resonance}.
\newblock {\em Physica D}, 137:333--352, 2000.

\bibitem{05}
M.~Freidlin and A.~Wentzell.
\newblock {\em {Random perturbations of dynamical systems}}.
\newblock Springer, 3rd edition, 2012.

\bibitem{06}
M.~Krein and D.~Milman.
\newblock {On extreme points of regular convex sets}.
\newblock {\em Studia Mathematica}, 9:133--138, 1940.

\end{thebibliography}

\end{document}